\documentclass{article}%
\usepackage{graphicx}
\usepackage{amsmath,amssymb,amsfonts}%
\usepackage{theorem}%
\usepackage{color}%
\usepackage{caption}
\usepackage{subcaption}
\usepackage{hyperref}
\usepackage[font={small, it}]{caption}
\usepackage{chngcntr}
\usepackage{overpic}

\theoremstyle{change}%
\newtheorem{definition}{Definition:}[section]%
\newtheorem{proposition}[definition]{Proposition:}%
\newtheorem{theorem}[definition]{Theorem:}%
\newtheorem{lemma}[definition]{Lemma:}%
\newtheorem{corollary}[definition]{Corollary:}%
{\theorembodyfont{\rmfamily}\newtheorem{remark}[definition]{Remark:}}%
{\theorembodyfont{\rmfamily}}%

\newenvironment{proof}{{\bf Proof:}}
{\qquad \hspace*{\fill} $\Box$}%

\newcommand{\fg}{\mathfrak{g}}%
\newcommand{\fn}{\mathfrak{n}}%
\newcommand{\tr}{\operatorname{tr}}%

\newcommand{\rme}{\mathrm{e}}%
\newcommand{\N}{\mathbb{N}}%
\newcommand{\R}{\mathbb{R}}%
\newcommand{\Z}{\mathbb{Z}}%

\counterwithin{equation}{section}

\begin{document}

\title{Maxwell Strata in the sub-Riemannian problem on solvable, nonnilpotent regular three-dimensional Lie groups}
 \author{Adriano Da Silva\thanks{A. Da Silva was supported by S\~ao Paulo Research Foundation FAPESP grant 2024/13683-3} \; and Margarita Quispe Tusco\thanks{M.Tusco was partially support by CAPES (Brazil) grant ``Move las Americas'' and  Programa de Doctorado en Ciencias con Menci\'on en Matem\'atica,
Departamento de Matem\'atica, Universidad de
Tarapac\'a, Arica, Chile,} \\
		Departamento de Matem\'atica,\\Universidad de Tarapac\'a - Chile.
 		\and
 		Lino Grama\thanks{L. Grama research is partially supported by S\~ao Paulo Research Foundation FAPESP grants 2018/13481-0, 2023/13131-8, and CNPq grant no.306021/2024-2.} 
        \; and Douglas Duarte Novaes\thanks{D.D. Novaes was partially supported by the São Paulo Research Foundation (FAPESP), Grant No. 2024/15612-6; by the Conselho Nacional de Desenvolvimento Científico e Tecnológico (CNPq), Grant No. 301878/2025-0; and by the Coordenação de Aperfeiçoamento de Pessoal de Nível Superior - Brasil (CAPES), through the MATH-AmSud program, Grant No. 88881.179491/2025-01.} \\
 		Universidade Estadual de Campinas, Brazil\\
      Instituto de Matemática, Estatística e Computação Científica
 }
 \date{\today }

 \maketitle

 \begin{abstract}
In this paper, we study the sub-Riemannian problem associated with contact structures on connected, simply connected, solvable, non-nilpotent, regular three-dimensional Lie groups. For these groups, the vertical component of the Hamiltonian system takes the form of a perturbed pendulum. A qualitative phase-space analysis allows us to prove that this vertical component exhibits nontrivial symmetries. In particular, we are able to fully characterize the Maxwell set corresponding to these symmetries, and show that its first Maxwell time coincides with the period of the pendulum for almost all geodesics. This result yields an explicit upper bound for the cut time in terms of the period of the pendulum.
 \end{abstract}

	

\section{Introduction}
In the last decades, the study of sub-Riemannian geometry from a control-theoretic perspective has attracted growing interest, leading to a number of significant contributions in the field (see for instance the \cite{AADBUB, AAYS, RM} and references therein). Among the topics of particular relevance is the full description of the cut and conjugate loci, which has been extensively investigated in the setting of contact structures on three-dimensional manifolds, specially on three-dimensional Lie groups. These cases have received special attention due to both the richness of their underlying structure and the possibility of obtaining explicit expressions for normal extremal trajectories.

Indeed, complete optimal syntheses have been derived on the Heisenberg group \cite{RB, VAGV0}, on the semisimple groups $\mathrm{SO}(3), \mathrm{SU}(2)$ and $\mathrm{SL}(2)$ \cite{IYBYS, UBFR}, on the group of motions of the plane $\mathrm{SE}(2)$ \cite{Sach1, Sach2, Sach5}, on the group of motions on the pseudo-Euclidean plane \cite{ButtSach0, ButtSach1, ButtSach2}, as well as in other related settings \cite{AADB, AAAYS1, AAAYS2, Sach3, Sach4}. These works provide deep and elegant results concerning the structure of cut and conjugate loci. 

Due to the high degree of symmetry in these groups, one of the main tools employed to study the loss of optimality of geodesics is the analysis of Maxwell points associated with discrete symmetries of the system. Such points play a crucial role, since extremal trajectories lose optimality after the first Maxwell point. Consequently, the corresponding Maxwell time naturally provides an upper bound for the cut time.

The goal of the present work is to contribute to this line of research by analyzing solvable, non-nilpotent, regular three-dimensional Lie groups, that is, groups given by semi-direct product of $\R\times_{\rho}\R^2$, where $\rho_t:=\rme^{t\theta}$ and the $2\times 2$ matrix $\theta$ satisfying $\det\theta\cdot\tr\theta\neq 0$. For this class of groups, the vertical part of the Hamiltonian system associated with the optimization problem reduces to a nontrivial perturbed pendulum. Instead of seeking closed-form expressions for the normal extremal trajectories, as in \cite{ButtSach0, ButtSach1, ButtSach2, Sach1, Sach2}, we carry out a qualitative study of the phase space. In particular, we show that the system admits nontrivial symmetries and that the phase space of the pendulum-type solutions can be decomposed into regions by heteroclinic orbits, in analogy with the classical pendulum. Furthermore, we study several properties of the period of the solutions and show that, in the finite case, the period can be characterized by the times at which the solutions intersect specific fibers of the cylinder. Building on these results, we carry out an analysis of the Maxwell sets associated with these symmetries. In particular, we are able to fully characterize the Maxwell set and determine the first Maxwell time in terms of the pendulum period for almost all trajectories. Finally, we establish that this same period provides an upper bound for the cut time on the whole cylinder.

The paper is organized as follows. Section 2 introduces the necessary preliminaries. We present the framework for the three-dimensional Lie groups under consideration, define the sub-Riemannian structure, and formulate the optimization problem of interest. Several auxiliary lemmas and propositions are also proved in this section.

In Section 3, we analyze the vertical component of the Hamiltonian system associated with the optimization problem. We carry out a qualitative analysis of the equation, first establishing the existence of homoclinic orbits connecting the saddle points. This result allows us to decompose the phase space of the system, drawing an analogy with classical results for the simple pendulum. We then study the main properties of the period of the pendulum solutions, including its continuous extension to the entire cylinder. For finite values, we show that the period can be characterized geometrically as the minimum time a solution requires to cross two prescribed values.

Subsequently, we investigate the symmetries of the pendulum solutions. We prove that the perturbed pendulum still admits nontrivial symmetries and, although an explicit formula for the normal extremal trajectories is not available, their integral representations allow us to demonstrate that the symmetries of the pendulum induce corresponding symmetries in the trajectories. This analysis sets the stage for Section 4, where we study the Maxwell set associated with these symmetries. In particular, we show that Maxwell points correspond to the zeros of the horizontal component. As a result, we prove that the $\R^2$-component of a normal extremal trajectory has no zeros, while the zeros of its $\R$-component can be fully characterized.

These characterizations lead to our main results: we obtain explicit expressions for the Maxwell set and for the first Maxwell time in terms of the period of the pendulum trajectories. As a consequence, we establish that the first Maxwell time is invariant under both the pendulum dynamics and the symmetry actions. Finally, we conclude the paper by proving that the pendulum period provides an upper bound for the cut time of normal extremal trajectories.

\section{Preliminaries}

In this section we introduce the main results and notations needed for the rest of the paper.

\subsection{3D Solvable, Nonnilpotent Lie groups}

In what follows, we introduce a representation for connected, simply connected, solvable, non-nilpotent three-dimensional Lie groups. This framework allows us to study several of their properties through the dynamics of $2 \times 2$ matrices.  

According to \cite[Theorem 1.4, Chapter 7]{onis}, every real three-dimensional Lie algebra that is solvable and non-nilpotent can be written as a \textit{semidirect product} of $\mathbb{R}$ with $\mathbb{R}^2$, denoted by
\[
\mathfrak{g}(\theta) := \mathbb{R} \times_{\theta} \mathbb{R}^2,
\]
where $\theta(t) := t\theta$ for a fixed $2 \times 2$ real matrix $\theta$, which, up to isomorphism, can be chosen in one of the following canonical forms:
\[
\theta \in \left\{ 
\begin{pmatrix} 1 & 1 \\ 0 & 1 \end{pmatrix}, \hspace{.5cm} 
\begin{pmatrix} 1 & 0 \\ 0 & \gamma \end{pmatrix} \; \Big|\; |\gamma| \leq 1, \hspace{.5cm} 
\begin{pmatrix} \gamma & -1 \\ 1 & \gamma \end{pmatrix} \; \Big|\; \gamma \in \mathbb{R} 
\right\}.
\]

The Lie bracket in $\mathfrak{g}(\theta)$ is completely determined by linearity and the rule
\[
[(1, 0), (0, \eta)] = (0, \theta \eta), \quad \text{for all } \eta \in \mathbb{R}^2.
\]

The corresponding connected, simply connected Lie group $G(\theta)$ associated with $\mathfrak{g}(\theta)$ is, up to isomorphism, given by the semidirect product
\[
G(\theta) := \mathbb{R} \times_{\rho} \mathbb{R}^2,
\]
where the representation $\rho: \mathbb{R} \to \mathrm{GL}(2, \mathbb{R})$ is defined by
\[
\rho_t := \mathrm{e}^{t\theta}.
\]
The group operation is then given by
\[
(z_1, w_1)(z_2, w_2) = \big(z_1 + z_2,\, w_1 + \rho_{z_1} w_2\big),
\]
for all $(z_1, w_1), (z_2, w_2) \in \mathbb{R} \times_{\rho} \mathbb{R}^2$. In this setting, a left-invariant vector field has the form
\[
X(z, w) = (\sigma, \rho_z \eta), \qquad \text{with} \qquad X(0, 0) = (\sigma, \eta) \in \mathfrak{g}(\theta).
\]

Moreover, the automorphism groups of $\mathfrak{g}(\theta)$ and $G(\theta)$ are given, respectively, by (see \cite[Propositions 2.1 and 2.2]{VAASDH})
\[
\mathrm{Aut}(\mathfrak{g}(\theta)) = 
\left\{ 
\begin{pmatrix}
\epsilon & 0 \\
\eta & P
\end{pmatrix} 
: \eta \in \mathbb{R}^2, \, P \in \mathrm{GL}(\mathbb{R}^2), \, P\theta = \epsilon \theta P
\right\},
\]
and
\[
\mathrm{Aut}(G(\theta)) = \Big\{ \phi(z, w) := \big(\epsilon z,\, Pw + \Lambda_{\epsilon z}^{\theta}\eta\big) : \eta \in \mathbb{R}^2, \, P \in \mathrm{GL}(\mathbb{R}^2), \, P\theta = \epsilon \theta P \Big\},
\]
where $\epsilon = 1$ if $\mathrm{tr}\theta \neq 0$, and $\epsilon \in \{-1, 1\}$ if $\mathrm{tr}\theta= 0$, and the operator $\Lambda:\R\times \R^2\rightarrow \R^2$ is defined by
    $$\Lambda_z^{\theta}\eta := \int_0^z \rho_t \eta \,\mathrm{d}t = \int_0^z \mathrm{e}^{t\theta} \eta \,\mathrm{d}t.$$

\begin{definition}
\label{definition}
The group $G(\theta)$ and/or its algebra $\fg(\theta)$ are said to be \textbf{regular} if $\det\theta \cdot \mathrm{tr}\theta \neq 0$.
\end{definition}


\subsection{Sub-Riemannian structures on $G(\theta)$}
A $2$-dimensional left-invariant distribution on the group $G(\theta)$ is a map 
\[
\Delta^L: G(\theta) \longrightarrow TG(\theta), \qquad  
\Delta^L(z, w) = (dL_{(z, w)})_{(0,0)}\Delta,
\]
where $\Delta \subset \mathfrak{g}(\theta)$ is a 2-dimensional vector subspace.  
If $\langle \cdot, \cdot \rangle$ is an inner product on $\Delta$, we equip $\Delta^L$ with a left-invariant Euclidean metric by setting
\[
\langle X, Y\rangle_{(z, w)} := 
\langle (dL_{(z, w)^{-1}})_{(z, w)}X,\; (dL_{(z, w)^{-1}})_{(z, w)}Y \rangle,
\qquad X, Y \in T_{(z, w)}G(\theta), \ (z, w)\in G(\theta).
\]

\begin{definition}
A \textbf{sub-Riemannian} (or \textbf{contact}) structure on $G(\theta)$ is a bracket-generating, $2$-dimensional left-invariant distribution endowed with an inner product.
\end{definition}

Let $\Delta$ be a $2$-dimensional subspace of $\mathfrak{g}(\theta)$, and let $\mathfrak{n}$ denote the nilradical of $\mathfrak{g}(\theta)$, that is, $\mathfrak{n} = \{0\}\times\mathbb{R}^2$.  
Since $\mathfrak{g}(\theta)$ is three-dimensional, we have
\[
\Delta \not\subset \mathfrak{n}
\quad \iff \quad \dim(\Delta \cap \mathfrak{n}) = 1.
\] 

If $\eta \in \mathbb{R}^2$ is a nonzero vector with $(0, \eta)\in \Delta \cap \mathfrak{n}$, then by \cite[Proposition 3.5]{VAASDH},
\[
\Delta \ \text{is a subalgebra} 
\quad \iff \quad \omega(\theta \eta, \eta) = 0,
\]
and consequently,
\[
\Delta \ \text{is bracket-generating} 
\quad \iff \quad \omega(\theta \eta, \eta) \neq 0,
\]
where $\omega(\eta, \xi) := \det(\eta \mid \xi)$ is the determinant of the matrix with columns $\eta$ and $\xi$, i.e. the unique (up to a nonzero scalar) nondegenerate skew-symmetric bilinear form on $\mathbb{R}^2$.

For any $\eta \in \mathbb{R}^2$ with $\omega(\theta \eta, \eta) \neq 0$, we denote by $\Delta_{\eta}$ the $2$-dimensional sub-Riemannian structure for which $\{(1,0), (0,\eta)\}$ is an orthonormal basis.  
The following result shows that, up to isometries and rescaling, $\Delta_{\eta}$ is in fact the unique $2$-dimensional sub-Riemannian structure on $\mathfrak{g}(\theta)$.

\begin{lemma}
\label{unique}
Let $\eta\in\R^2$ be a nonzero vector satisfying $\omega(\eta, \theta\eta)\neq 0$. Then, up to isometries and rescaling, $\Delta_{\eta}$ is the only $2$-dimensional sub-Riemannian structure on $G(\theta)$.
\end{lemma}

\begin{proof}
Let $\Delta$ be a two-dimensional sub-Riemannian structure on $G(\theta)$ and let 
$\xi\in\R^2$ be a nonzero vector such that $(0,\xi)\in \Delta\cap\fn$. By the bracket-generating condition we must have 
\[
\omega(\theta\xi,\xi)\neq 0 \quad\text{and}\quad \R\cdot (0,\xi)=\Delta\cap\fn.
\]

Since $\omega(\theta\xi,\xi)\neq 0$ and $\omega(\theta\eta,\eta)\neq 0$, the sets 
$\{\xi,\theta\xi\}$ and $\{\eta,\theta\eta\}$ are both bases of $\R^2$. In particular, by the Cayley--Hamilton Theorem, we obtain
\[
\theta^2\xi = (-\det\theta)\cdot \xi + (\tr\theta)\cdot\theta\xi, 
\qquad 
\theta^2\eta = (-\det\theta)\cdot \eta + (\tr\theta)\cdot\theta\eta.
\]

Let $P\in\mathfrak{gl}(\R^2)$ be the linear map defined by $P\xi=\eta$ and $P\theta\xi=\theta\eta$. Then $\det P\neq 0$ and $P\in\mathrm{Gl}(\R^2)$. Also, using the identities above, we see that
\[
P\theta\xi = \theta\eta = \theta P\xi, \qquad 
P\theta^2\xi = \theta^2\eta = \theta P\theta\xi,
\]
which shows that $P\theta$ and $\theta P$ coincide on the basis $\{\xi,\theta\xi\}$. Hence, $P\theta=\theta P$. 

Now choose a vector $(\sigma,\xi')\in\fg(\theta)$, with $\sigma>0$, orthogonal to $(0,\xi)$ in $\Delta$, and define the automorphism
\[
\psi(z,w)=\left(z, Pw - \tfrac{1}{\sigma}\Lambda^{\theta}_z P\xi'\right),
\qquad 
(d\psi)_{(0,0)}=
\begin{pmatrix}
1 & 0 \\
-\tfrac{1}{\sigma}P\xi' & P
\end{pmatrix}.
\]
Since
\[
(d\psi)_{(0,0)}(0,\xi)=(0,\eta),
\qquad
(d\psi)_{(0,0)}(\sigma,\xi')=(\sigma,0),
\]
the subspace $\Delta'=(d\psi)_{(0,0)}\Delta$ can be endowed with an inner product such that $\psi$ is an isometry. With this inner product, $\{(\sigma,0),(0,\eta)\}$ is an orthogonal basis. After a rescaling, we obtain the orthonormal basis $\{(1,0),(0,\eta)\}$ for $\Delta'$, i.e., $\Delta'=\Delta_{\eta}$. This completes the proof.
\end{proof}

\subsection{The Pontryagin Maximum Principle}

Fix $\eta \in \R^2$ such that $\omega(\eta, \theta\eta) \neq 0$, and consider, in the sense of Lemma \ref{unique}, the unique sub-Riemannian structure on $G(\theta)$ defined by
$$\Delta_{\eta}(g) = \mathrm{span}\{X_1(g), X_2(g)\}, \hspace{.5cm} \langle X_i, X_j \rangle = \delta_{i, j}, \hspace{.5cm} i, j = 1, 2,$$
$$X_1(g) = (1, 0), \hspace{.5cm} X_2(g) = (0, \rho_z \eta), \hspace{.5cm} g = (z, w) \in G(\theta).$$

Associated with this choice, we obtain the following optimal control problem:
\begin{equation}
\label{1}
 \dot{g} = u_1 X_1(g) + u_2 X_2(g), \hspace{.5cm} g \in G(\theta), \hspace{.5cm} u = (u_1, u_2) \in \R^2,
\end{equation}
$$g(0) = g_0 = (0, 0), \hspace{.5cm} g(t_1) = g_1,$$
$$\ell = \int_0^{t_1} \sqrt{u_1^2 + u_2^2}\, dt \;\rightarrow\; \mathrm{min}.$$

In the coordinates $(z, w)$ this reads as
$$\dot{z} = u_1, \hspace{.5cm} \dot{w} = u_2 \rho_z \eta,$$
$$g = (z, w) \in G(\theta) = \R \times_{\rho} \R^2, \hspace{.5cm} u = (u_1, u_2) \in \R^2.$$
\begin{equation}
\label{2}
g(0) = g_0 = (0, 0), \hspace{.5cm} g(t_1) = g_1 = (z_1, w_1).
\end{equation}
$$\ell = \int_0^{t_1} \sqrt{u_1^2 + u_2^2}\, dt \;\rightarrow\; \mathrm{min}.$$

The admissible controls $u$ are bounded and measurable, while the admissible trajectories $g$ are Lipschitz continuous. Moreover, by the Cauchy--Schwarz inequality, the minimization of the sub-Riemannian length functional $\ell$ is equivalent to minimizing the energy functional
\begin{equation}
\label{3}
J = \frac{1}{2}\int_0^{t_1} (u_1^2 + u_2^2)\, dt \;\rightarrow\; \mathrm{min}.
\end{equation}

Since
$$X_3 := [X_1, X_2] = [(1, 0), (0, \eta)] = (0, \theta \eta),$$
we have
$$\mathrm{span}\{X_1(g), X_2(g), X_3(g)\} = T_g G(\theta).$$
Thus, the system is full rank and therefore completely controllable on $G(\theta)$. Furthermore, by Filippov’s theorem, the existence of optimal controls for the problem (\ref{1}), (\ref{2}), (\ref{3}) is guaranteed (see, e.g., \cite[Chapter 10]{AAYS}).

Now consider the cotangent bundle $T^*G(\theta)$ of $G(\theta)$, with $\pi : T^*G(\theta) \to G(\theta)$ denoting the canonical projection onto the base. For a Hamiltonian function $h \in C^{\infty}(T^*G(\theta))$, let $\Vec{h} \in \mathrm{Vec}(T^*G(\theta))$ denote its associated Hamiltonian vector field. 

Define the Hamiltonian functions, linear on the fibers, as  
$$h_i(\lambda) = \langle \lambda, X_i(g) \rangle, \hspace{.5cm} i = 1, 2, 3, \hspace{.5cm} \pi(\lambda) = g,$$
and the control-dependent Hamiltonian of the PMP as
$$h_u^{\nu}(\lambda) = u_1 h_1(\lambda) + u_2 h_2(\lambda) + \frac{\nu}{2}(u_1^2 + u_2^2),$$
where $\lambda \in T^*G(\theta)$, $u = (u_1, u_2) \in \R^2$, and $\nu \in \{-1, 0\}$. 

With these notations, the Pontryagin Maximum Principle for the problem under consideration can be stated as follows:

\begin{theorem}[Pontryagin Maximum Principle]
Let $u(t)$ and $g(t)$, $t \in [0, t_1]$, be an optimal control and its corresponding optimal trajectory in the problem (\ref{1}), (\ref{2}), (\ref{3}). Then there exist a Lipschitz continuous curve $\lambda(t) \in T^*G(\theta)$ with $\pi(\lambda(t)) = g(t)$ for all $t \in [0, t_1]$, and a number $\nu \in \{-1, 0\}$ such that, for almost every $t \in [0, t_1]$, the following conditions hold:
$$\dot{\lambda}(t) = \Vec{h}_{u(t)}^{\nu}(\lambda(t)) = u_1(t)\Vec{h}_1(\lambda(t)) + u_2(t)\Vec{h}_2(\lambda(t)),$$
$$h_{u(t)}^{\nu}(\lambda(t)) = \max_{u \in \R^2} h_u^{\nu}(\lambda(t)),$$
$$(\nu, \lambda(t)) \neq 0.$$
\end{theorem}

Since the optimization problem is contact, the abnormal case $\nu=0$ admits only constant optimal trajectories. Hence, all optimal trajectories are strictly normal (see, for instance, \cite[Proposition 4.38]{AADBUB}). Therefore, we restrict our attention to the normal case $\nu=-1$. In this setting, the maximality condition implies that normal extremals satisfy
$$u_i(t)=h_i(\lambda(t)), \hspace{.5cm}i=1,2,\; t\in [0, t_1].$$
Consequently, $\lambda(t)$ is a solution of the maximized Hamiltonian system
\begin{equation}
    \label{maximized}
    \dot{\lambda}=\Vec{H}(\lambda), \hspace{.5cm}\lambda\in T^*G(\theta),
\end{equation}
where $H(\lambda)=(h_1(\lambda)^2+h_2(\lambda)^2)/2$.  

To simplify the analysis, we reduce the problem by one dimension by restricting to the level surface $H=1/2$. This reduction yields convenient expressions for the vertical and horizontal components of the maximized Hamiltonian.  

In fact, write as previously
$$\theta^2\eta=-\det\theta\cdot\eta+\tr\theta\cdot\theta\eta.$$
Thus,  
$$[X_1,X_2]=X_3, \hspace{.5cm}[X_2,X_3]=0, \hspace{.5cm}[X_1,X_3]=-\det\theta \cdot X_2+\tr\theta \cdot X_3.$$

Since $\{h_1,h_2,h_3\}$ provides a system of vertical coordinates on $T^*G(\theta)$, the Hamiltonian equations (\ref{maximized}) becomes, in coordinates,
$$\dot{h}_1=-h_2h_3, \hspace{.5cm}\dot{h}_2=h_1h_3, \hspace{.5cm}\dot{h}_3=(-\det\theta\cdot h_2+\tr\theta\cdot h_3)h_1,$$
$$\dot{z}=h_1, \hspace{.5cm}\dot{w}=h_2\rho_z\eta.$$

Restricting to trajectories on the level surface $H=1/2$ (corresponding to arc-length parametrized extremals),
$$C=T^*_{(0,0)}G(\theta)\cap H^{-1}\left(\frac{1}{2}\right),$$
allows us to introduce the polar coordinates
$$h_1=\cos\varphi, \hspace{.5cm} h_2=\sin\varphi, \hspace{.5cm} h_3=r.$$
This yields the following result:

\begin{proposition}
\label{equations}
The vertical part of the maximized Hamiltonian system (\ref{maximized}) of the sub-Riemannian structure $\Delta_{\eta}$ on $G(\theta)$ takes the form of a perturbed pendulum:
\begin{equation}
\label{verticaldouble}
\dot{\varphi}=r, \hspace{.5cm}\dot{r}=-\tfrac{\det\theta}{2}\sin(2\varphi)+\tr\theta\,r\cos\varphi,\hspace{.5cm} (\varphi, r)\in \mathbb{S}^1\times\R.
\end{equation}
The corresponding horizontal dynamics is given by
\begin{equation}
\label{horizontaldouble}
\dot{z}=\cos\varphi, \hspace{.5cm} \dot{w}=\sin\varphi\cdot\rho_z\eta.
\end{equation}
\end{proposition}

Normal extremals, parametrized by arc length, are described by the \textbf{exponential map}
$$\mathrm{Exp}: C\times \R_+\rightarrow G(\theta), \hspace{.5cm} \mathrm{Exp}(\lambda,T):=\pi(\lambda(T)).$$
By Proposition \ref{equations}, we have $\lambda(t)=((\varphi(t), r(t)),(z(t),w(t)))$, where $(\varphi(t), r(t))$ solve the vertical system (\ref{verticaldouble}) with initial condition $\lambda$, and $z(t),w(t)$ solve the horizontal system (\ref{horizontaldouble}) with initial condition $(0,0)$. In particular, the exponential map depends explicitly on the vertical dynamics. Since this dynamics reduces to a perturbed pendulum, no closed-form expression for the exponential map is available. Therefore, to obtain information about the behavior of the exponential map, we proceed in the next sections with a qualitative analysis of the solutions of (\ref{verticaldouble}).  

\begin{remark}
It is worth noting that expression (\ref{verticaldouble}) for the vertical dynamics was previously obtained, up to a phase shift of $\pi/2$, in \cite[Section 5]{VAGV} by analyzing the structure constants of the Lie algebra $\fg(\theta)$. However, an additional dynamical behavior of the solutions of (\ref{verticaldouble}) was overlooked in \cite{VAGV} and will play a crucial role in our subsequent analysis (see Proposition \ref{heteroclinic}). 
\end{remark}

\section{The Vertical Part of the Hamiltonian System}

In this section, we study the vertical component of the Hamiltonian system obtained in (\ref{verticaldouble}).  
Observe that the equations in (\ref{verticaldouble}) are equivalent to the second-order ODE
\begin{equation}
\label{HamiltonianEquivalent}
    \Ddot{\varphi}=-\tfrac{\det\theta}{2}\sin(2\varphi)+\tr\theta \,\dot{\varphi}\cos\varphi,\hspace{.5cm} (\varphi(0), \dot{\varphi}(0))\in \mathbb{S}^1\times \R,
\end{equation}
or the differential system
\begin{equation}
\label{HamiltonianEquivalent1}
\begin{cases}
\dot{\varphi}=r,\hspace{.5cm}\\
    \dot{r}=-\tfrac{\det\theta}{2}\sin(2\varphi)+\tr\theta \,r\,\cos\varphi,
    \end{cases}\quad(\varphi, r)\in \mathbb{S}^1\times \R,
\end{equation}
whose dynamics we now proceed to analyze. In particular, we focus on the \emph{regular case}, defined by the condition $\det\theta\cdot\tr\theta\neq 0$.  

It is worth noting that a complete characterization of certain non-regular cases has already appeared in the literature. Specifically, such results were obtained in \cite{MoSach, Sach1, Sach2} for the group of right motions $SE(2)$, which appears as a quotient of $G(\theta)$, where
\[
\theta=\begin{pmatrix}
    0 & -1 \\
    1 & 0
\end{pmatrix},
\]
and in \cite{ButtSach1, ButtSach2} for the group of motions of the pseudo-Euclidean plane $SH(2)=G(\theta)$, where
\[
\theta=\begin{pmatrix}
    1 & 0 \\
    0 & -1
\end{pmatrix}.
\]
The essential difference in our setting lies in the fact that, in those earlier cases, the perturbation term vanishes because 
$\tr\theta=0$, and the authors exploit this fact to integrate equation 
(\ref{HamiltonianEquivalent}) in terms of elliptic functions and to analyze the solutions explicitly. By contrast, in our setup the perturbation term satisfies $\tr\theta\neq 0$, and our approach relies on a qualitative analysis of the integral form of the solutions and of the properties of the period function. This allows us to treat all regular groups simultaneously.

\subsection{Decomposition of the phase-space}

The previous system admits singularities at the points 
$$\left(\frac{k\pi}{2}, 0\right), k\in\Z.$$
At such points, we have the associated Jacobian matrix
\begin{equation}\label{jacobian}
    J_k:=J\left(\frac{k\pi}{2}, 0\right)=\left(\begin{array}{cc}
    0& 1 \\
   (-1)^{k+1}\det\theta  &  \tr\theta \cos\left(\frac{k\pi}{2}\right) 
\end{array}\right),
\end{equation}
whose characteristic polynomial and discriminant are
$$p(\lambda)=\lambda^2-\tr\theta \cos\left(\frac{k\pi}{2}\right) \lambda+(-1)^k\det\theta\hspace{.5cm}\mbox{ and }\hspace{.5cm}\Delta_k=(\tr\theta)^2\cos^2\left(\frac{k\pi}{2}\right)-4(-1)^k\det\theta.$$
Since,
$$\cos\left(\frac{k\pi}{2}\right)=0\hspace{.5cm}\iff\hspace{.5cm}
    k\mbox{ is odd},$$
we get that
\begin{itemize}
    \item $\det\theta>0$: In this case, $\Delta_k>0$ for $k=-1, 1$ and the points $(\pm\pi/2, 0)$ are saddles. The points $(0, 0)$ and  $(\pi, 0)$ attractors and repellers, or vice-versa, depending of the sign of $\tr\theta$. 

    \item $\det\theta<0$: In this case, $\Delta_k<0$ for $k=-1, 1$ and the points $(\pm\pi/2, 0)$ can be centers or weak focus (we will see that they are centers). The points $(0, 0)$ and  $(\pi, 0)$ saddles.
\end{itemize}

\begin{proposition}
\label{heteroclinic}
Consider system \eqref{HamiltonianEquivalent} defined on the cylinder $\mathbb{S}^1\times \R$, with $\mathbb{S}^1=[-\pi,\pi]/\sim$, where $\sim$ identifies $-\pi$ and $\pi$. In this case, one has the following set equilibria $$\{p_1=(-\pi,0),p_2=(-\pi/2,0),p_3=(0,0),p_4=(\pi/2,0)\},$$ where $(-\pi,0)$ and $(\pi,0)$ coresponds to the same point $p_1$.
\begin{itemize}

\item[(a)] Assume $\det \theta < 0$. Then $p_1$ and $p_3$ are saddles, while $p_2$ and $p_4$ are centers. In this case, there exist four heteroclinic orbits:  
\begin{itemize}
    \item $H^+_{1,3}\subset (-\pi,0)\times \mathbb{R}_{>0}$, connecting $p_1$ to $p_3$.
    \item $H^+_{3,1}\subset (0,\pi)\times \mathbb{R}_{>0}$, connecting $p_3$ to $p_1$.
    \item $H^-_{1,3}\subset (-\pi,0)\times \mathbb{R}_{<0}$, connecting $p_1$ to $p_3$.
    \item $H^-_{3,1}\subset (0,\pi)\times \mathbb{R}_{<0}$, connecting $p_3$ to $p_1$.
\end{itemize}

Apart from the equilibria and heteroclinic connections, all other trajectories are periodic. These periodic orbits are organized into four period annuli: two bounded ones and two unbounded ones.  In $[-\pi,0]\times \mathbb{R}$, the set  
$
\gamma_1 = H^+_{1,3}\cup H^-_{3,1}\cup\{p_1,p_3\}
$
forms a simple closed curve (contractible in the cylinder), which corresponds to the outer boundary of the period annulus surrounding the center $p_2$. Similarly, in $[0,\pi]\times \mathbb{R}$, the set  
$
\gamma_3 = H^+_{3,1}\cup H^-_{1,3}\cup\{p_1,p_3\}
$
forms a simple closed curve (contractible in the cylinder), which corresponds to the outer boundary of the period annulus surrounding the center $p_4$. Finally, $\gamma^+=H^+_{1,3}\cup H^+_{3,1}\cup\{p_1,p_3\}$ forms a simple closed curve (non-contractible in the cylinder) that constitutes the inner boundary of the unbounded period annulus contained in $(-\pi,\pi)\times \mathbb{R}_{>0}$. Similarly, $\gamma^-=H^-_{1,3}\cup H^-_{3,1}\cup\{p_1,p_3\}$ forms a simple closed curve (non-contractible in the cylinder) that constitutes the inner boundary of the unbounded period annulus contained in $(-\pi,\pi)\times \mathbb{R}_{<0}$ (see Figure~\ref{Figure1}).

\item[(b)] Assume $\det \theta > 0$. Then, $p_2$ and $p_4$ are saddle points, while $p_1$ and $p_3$ are both foci when $(\tr\theta)^2 - 4\det\theta < 0$, or both nodes when $(\tr\theta)^2 - 4\det\theta > 0$. Moreover, $p_1$ is attracting (resp. repelling) and $p_3$ is repelling (resp. attracting) whenever $\tr\theta > 0$ (resp. $\tr\theta < 0$). Furthermore, there exist two homoclinic orbits:
\begin{itemize}
    \item $H_2$ connecting $p_2$ to itself, which is contained within $\mathbb{S}^1 \times \mathbb{R}_{>0}$ (resp. $\mathbb{S}^1 \times \mathbb{R}_{<0}$) provided that $\tr\theta>0$ (resp. $\tr\theta<0$).
    \item $H_4$ connecting $p_4$ to itself, which is contained within $\mathbb{S}^1 \times \mathbb{R}_{<0}$ (resp. $\mathbb{S}^1 \times \mathbb{R}_{>0}$) provided that $\tr\theta>0$ (resp. $\tr\theta<0$).
\end{itemize}
In addition, let $\gamma_2 = H_2 \cup \{p_2\}$ and $\gamma_4 = H_4 \cup \{p_4\}$, which are simple closed curves (non-contractible in the cylinder), and let $K$ denote the compact region bounded by $\gamma_2$ and $\gamma_4$. Then, all trajectories in $K$ are non-periodic and, apart from the homoclinic connections and equilibria, each trajectory converges to the attracting focus/node as $t \to +\infty$, and to the repelling focus/node as $t \to -\infty$. Finally, all trajectories lying outside the region $K$ are periodic (see Figure~\ref{Figure2}).


\end{itemize}
\end{proposition}

\begin{proof}
All conclusions follow from two key properties of the differential system \eqref{HamiltonianEquivalent1}: 

(A) Any orbit of \eqref{HamiltonianEquivalent1} restricted to $[-\pi+a,\pi+a]\times \mathbb{R}$, for any $a>0$, is bounded. This follows from the construction of suitable compact regions which, in view of the behavior of the vector field along its boundary, can be shown to confine the orbit to its interior.


(B) The differential system \eqref{HamiltonianEquivalent1} is reversible with respect to the lines $\varphi=\pi/2$ and $\varphi=-\pi/2$. 
Such a property implies that if $(\varphi(t), r(t))$ is a solution of \eqref{HamiltonianEquivalent1} satisfying $\varphi(0)=\pm\pi/2$, then
\begin{equation}\label{reverimpli}
\varphi(t)+\varphi(-t)=\pm\pi,
\end{equation}
for every $t\in\mathbb{R}$.  
Consequently, if a solution intersects the line $\varphi=\pi/2$ at two distinct times, or the line $\varphi=-\pi/2$ at two distinct times, or if it intersects both lines $\varphi=\pi/2$ and $\varphi=-\pi/2$, then the solution is periodic. Moreover, if a solution approaches asymptotically one of the lines $\varphi=\pi/2$ or $\varphi=-\pi/2$ and intersects the other, then it corresponds to a homoclinic connection.

In the case $\det\theta<0$, since the equilibrium point $p_2$ is monodromic, the orbits in its neighborhood are periodic. Therefore, by the Poincaré–Bendixson Theorem, the stable and unstable manifolds of the saddle point $p_1$ must either approach $p_3$ or intersect the line $\varphi = 0$, since they can neither converge to $p_2$ nor become unbounded before reaching $\varphi = 0$ because of (A). Analogously, the stable and unstable manifolds of the saddle point $p_3$ must either approach $p_1$ or intersect the line $\varphi = 0\equiv\pi$. Hence, by uniqueness of solutions, the only remaining possibility is that the unstable (respectively, stable) manifold of $p_1$ coincides with the stable (respectively, unstable) manifold of $p_3$, thereby establishing the existence of heteroclinic connections between the two saddle points that bound the period annulus surrounding $p_2$ and $p_4$ (see Figure \ref{Figure1}). Finally, since every solution outside $\gamma_1 \cup \gamma_3$ is bounded, the Poincaré–Bendixson Theorem implies that it must intersect the lines $\varphi = -\pi/2$ and $\varphi = \pi/2$, and hence corresponds to a periodic orbit.

In the case $\det \theta > 0$, the equilibria $p_2$ and $p_4$ are saddles, whereas $p_1$ and $p_3$ are both hyperbolic foci or nodes. Their stability is determined by the traces of the Jacobian matrices $J_1$ and $J_3$, respectively, given in \eqref{jacobian}. Assume $\tr \theta > 0$, the case $\tr \theta < 0$ can be treated analogously. Under this assumption, $p_1$ is attracting and $p_3$ is repelling. One branch of the unstable manifold of $p_2$ converges to $p_1$, while, by the Poincaré–Bendixson Theorem, the other branch must intersect the line $\varphi = \pi/2$. Since this trajectory approaches the line $\varphi = -\pi/2$ asymptotically, it follows from the consequences of \eqref{reverimpli} that there exists a homoclinic connection of $p_2$ to itself. An analogous argument applies to the equilibrium $p_4$ (see Figure~\ref{Figure2}). Finally, since every orbit outside $K$ restricted to $[-\pi/2,3\pi/2]\times\R$ is bounded and does not converge to any equilibria, the Poincaré–Bendixson Theorem implies that each such solution must intersect the lines $\varphi = -\pi/2$ and $\varphi = \pi/2$, and therefore corresponds to a periodic orbit.
\end{proof}



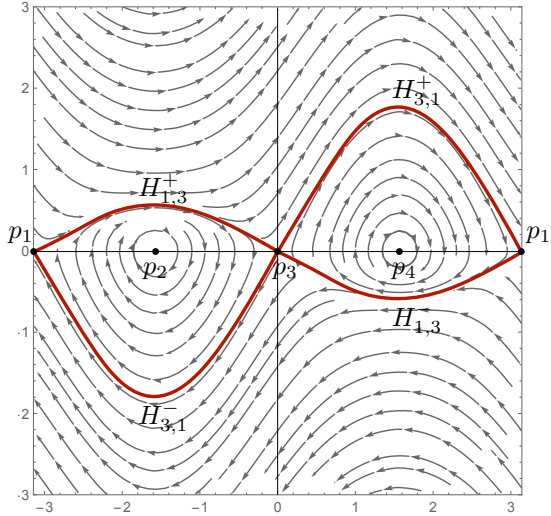
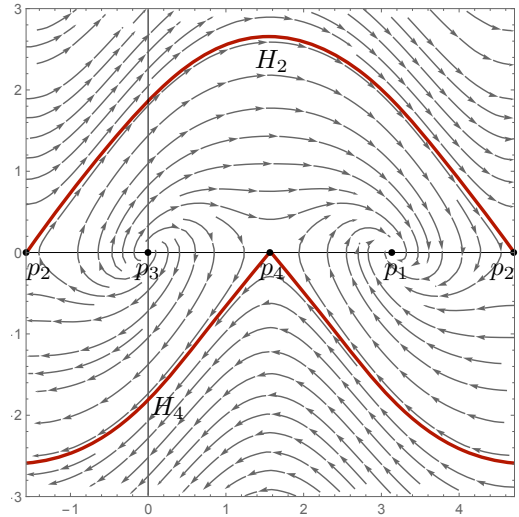
\begin{figure}[h!]
	\centering
	\begin{subfigure}{.5\textwidth}
		\centering
		\begin{overpic}
        [width=0.8\linewidth]{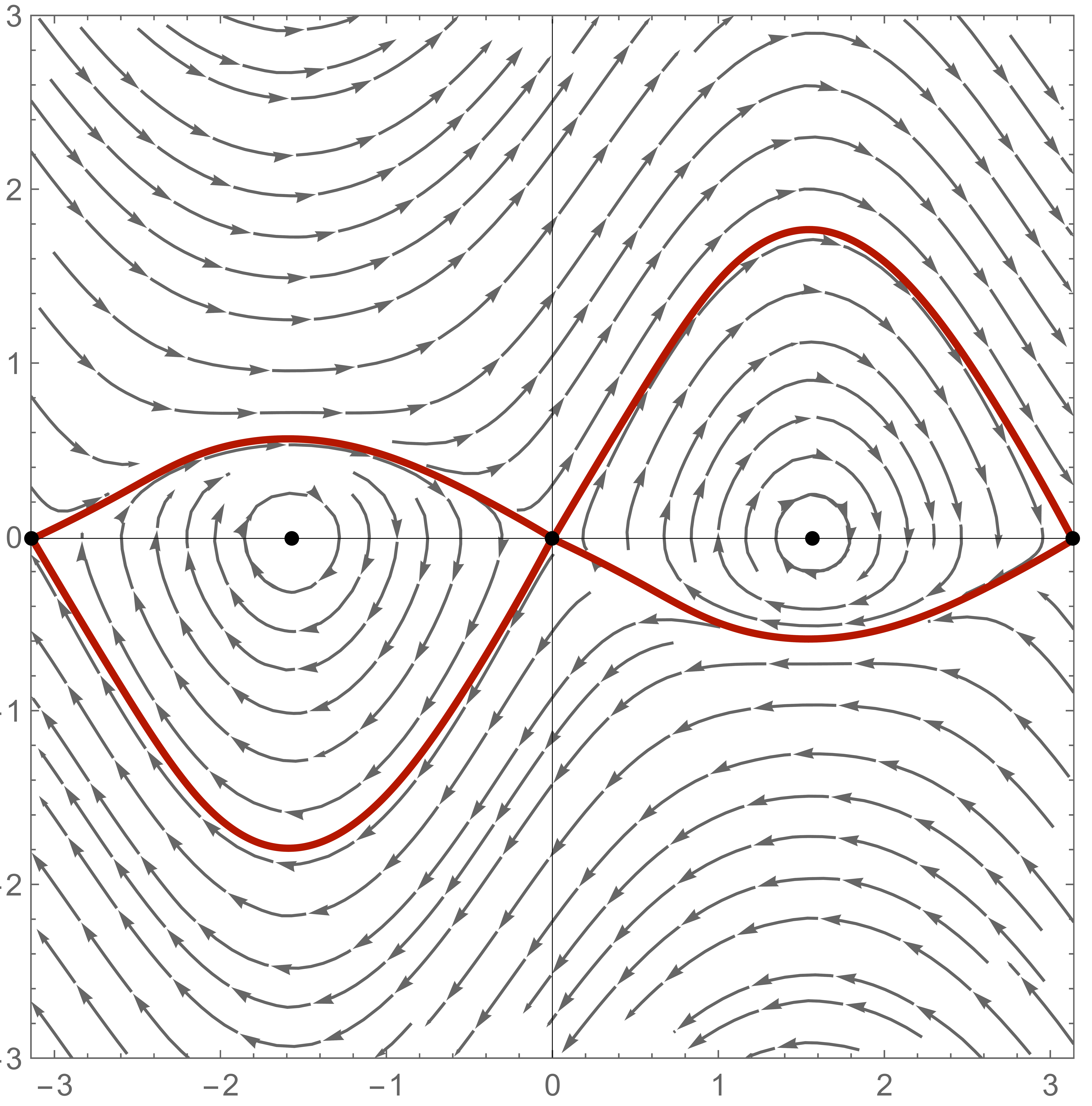}
        \put(-2,54){$p_1$}
        \put(24,47){$p_2$}
        \put(49,47){$p_3$}
        \put(72,47){$p_4$}
        \put(98,54){$p_1$}
        \put(23,62){$H_{1,3}^+$}
        \put(23,18){$H_{3,1}^-$}
        \put(72,37){$H_{1,3}^-$}
        \put(72,81){$H_{3,1}^+$}
        \end{overpic}
		\caption{Illustration of the case $\det\theta<0$; \vspace{.4cm}}
		\label{Figure1}
	\end{subfigure}%
	\begin{subfigure}{.5\textwidth}
		\centering
		\begin{overpic}
        [width=0.8\linewidth]{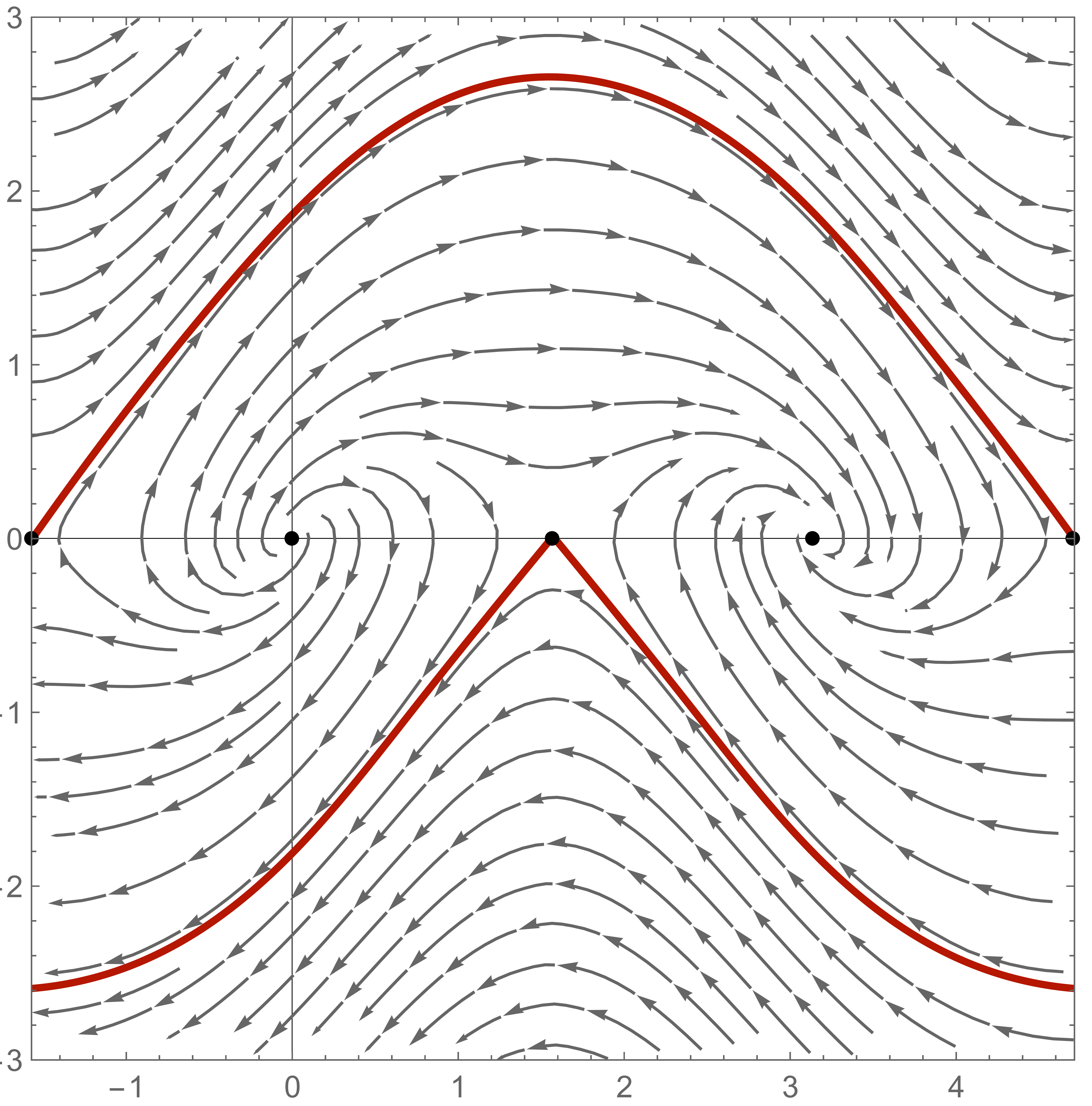}
         \put(3,47){$p_2$}
        \put(24,47){$p_3$}
        \put(48,47){$p_4$}
        \put(72,47){$p_1$}
        \put(92.5,47){$p_2$}
        \put(47,87){$H_2$}
        \put(27,20){$H_4$}
        \end{overpic}
		\caption{Illustration of the case $\det\theta>0$, where, for clarity, the phase space is shown for $\varphi \in (-\pi/2, 3\pi/2)$.}
		\label{Figure2}
	\end{subfigure}
	\caption{Phase space of system \eqref{HamiltonianEquivalent}.}
\end{figure}

\subsubsection{The period of the trajectories of the pendulum}

In this section, we study several properties of the periods of the solutions of (\ref{HamiltonianEquivalent}) that will be essential for determining the first Maxwell time. We begin with an analysis of the behavior of the period as one approaches the heteroclinic orbits and/or the equilibria of the system.

\begin{proposition}
Consider system \eqref{HamiltonianEquivalent} defined on the cylinder $C=\mathbb{S}^1\times \R$. Set $A_{\theta}:=C\setminus (\gamma_1\cup \gamma_3)$ for $\det\theta<0$, and $A_{\theta}:=C\setminus K$ for $\det\theta>0$.
Then, there exists a continuous function $\tau:A_{\theta}\to \R$ mapping each point of $A_{\theta}$, which is not an equilibrium of \eqref{HamiltonianEquivalent},  to the period of the solution of \eqref{HamiltonianEquivalent} passing through it. Moreover, such a function satisfies the following:
\begin{itemize} 

\item If $\det\theta<0$, then $\tau(p_2)=\tau(p_4)=2\pi/\sqrt{-\det\theta}$, $\tau(\lambda)\to \infty$ as $\lambda$ approaches $\gamma_1\cup \gamma_2$, and $\tau(\lambda)\to 0$ as $|\lambda|\to \infty$.

\item If $\det\theta>0$, then  $\tau(\lambda)\to +\infty$ as $\lambda$ approaches $\gamma_2\cup \gamma_4$ and $\tau(\lambda)\to 0$ as $|\lambda|\to \infty$.
\end{itemize}

\end{proposition}

\begin{proof}
The behavior of the function $\tau$ is analyzed considering transversal sections to the flow of \eqref{HamiltonianEquivalent}.

Assume that $\det \theta < 0$. The analysis for the case $\det \theta >0$ is analogous. To analyze the behavior of $\tau$ for $\lambda$ near $p_1$ and $p_2$, we introduce the following transversal sections
\[
\Sigma_i=\{\,p_i+(0, r)\;:\;r>0\,\}.
\]
The period of trajectories in a neighborhood of $p_1$ and $p_4$ can be studied using standard asymptotics analysis, which yields 
\[
\tau\bigl(p_i+(0, r)\bigr)
=\frac{2\pi}{\sqrt{-\det\theta}}\,r+\mathcal{O}(r^2).
\]
In particular, it follows that
\[
\tau(p_i)=\frac{2\pi}{\sqrt{-\det\theta}}, \qquad i=1,2.
\]

On the other hand, as a point $\lambda$ approaches the set $\gamma_1\cup \gamma_2$, the
corresponding trajectory spends an increasingly long time in a neighborhood of the saddle points $p_1$ and $p_3$ before completing one turn. Consequently, the period diverges, and we obtain
\[
\tau(\lambda)\to+\infty \hspace{.5cm} \text{ as } \hspace{.5cm} \lambda\to \gamma_1\cup \gamma_2.
\]

Finally, we study the asymptotic behavior of the period function $\tau$ as $|\lambda|\to\infty$ along the transversal section
\[
\Sigma=\{(0, r): r> r_0\}, \qquad \text{for $r_0$ sufficiently large}.
\]
To this end, we introduce the change of variables $y=1/R$. Under this transformation, system \eqref{HamiltonianEquivalent1} becomes
\begin{equation}\label{HamilInfinity1}
\begin{cases}
\dot{\varphi}=\dfrac{1}{R},\vspace{0.2cm}\\
\dot{R}=R^2\,\tfrac{\det\theta}{2}\sin(2\varphi)-\tr\theta \,R\,\cos\varphi.
\end{cases}
\end{equation}
Let $T(\varphi,R)$ denote the period function associated with \eqref{HamilInfinity1}. By construction, we have
\[
T(\varphi, R)=\tau(\varphi, 1/R).
\]

Next, performing a time rescaling, system \eqref{HamilInfinity1} can be rewritten as
\begin{equation}\label{HamilInfinity2}
\begin{cases}
\varphi'=1,\vspace{0.2cm}\\
R'=R^3\,\tfrac{\det\theta}{2}\sin(2\varphi)-\tr\theta \,R^2\,\cos\varphi,
\end{cases}
\end{equation}
which has the trivial period function $\widetilde T(\varphi, R)=2\pi$. Since systems \eqref{HamilInfinity1} and \eqref{HamilInfinity2} are related by a multiplicative factor $R$, the period function $T$ can be recovered from $\widetilde T$ as 
\[
T(0, R)=\int_0^{2\pi}\widetilde R(s, R)\,ds,
\]
where $s\mapsto (s,\widetilde R(s, R))$ denotes the solution of \eqref{HamilInfinity2} with initial condition $(0,R)$ (see, for instance, \cite[Proposition~1.14]{chicone}). Consequently, $T(0,0)=0$, which implies that $\tau(0, r)\to 0$ as $r\to\infty$.

\end{proof}

\begin{remark}
Since the solutions of (\ref{HamiltonianEquivalent}) starting at a point $\lambda\in C\setminus A_{\theta}$ are not periodic, we define their period to be infinite and write $\tau(\lambda)=+\infty$. By the previous proposition, the function $\tau:C\rightarrow (0, +\infty]$ defined in this way is continuous, and we will continue to refer to it as the period function.
\end{remark}



The next lemma states a simply property of the solutions of (\ref{HamiltonianEquivalent}) that will be very useful in our results ahead.

\begin{lemma}
\label{lemaT0}
    Let $\varphi\in \mathbb{S}^1$ be a solution of (\ref{HamiltonianEquivalent}) and assume that $2\varphi(T_0)=\pi\epsilon$ for some $T_0>0$, where $\epsilon\in\{-1, 1\}$. Then,  
    \begin{equation}
    \label{property}
        \forall t\in \R, \hspace{.5cm}\varphi(T_0-t)+\varphi(T_0+t)=\pi\epsilon.
    \end{equation}        
    \end{lemma}

\begin{proof}  Define the functions 
$$\alpha(t)=\varphi(T_0+t)\hspace{.5cm}\mbox{ and }\hspace{.5cm}\beta(t)=\pi\epsilon-\varphi(T_0-t),$$
and note that, 
$$\alpha(0)=\varphi(T_0)=\epsilon\frac{\pi}{2}=\left(\pi-\frac{\pi}{2}\right)\epsilon=\beta(0)\hspace{.5cm}\mbox{ and }\hspace{.5cm}\dot{\alpha}(0)=\dot{\varphi}(T_0)=\dot{\beta}(0).$$
Furthermore, 
$$\Ddot{\beta}(t)=-\Ddot{\varphi}(T_0-t)=-\Bigl\{-\tfrac{\det\theta}{2}\sin2\varphi(T_0-t)+\tr\theta\dot{\varphi}(T_0-t)\cos\varphi(T_0-t)\Bigr\}$$
$$=\tfrac{\det\theta}{2}\sin(2(\pi\epsilon-\beta(t)))-\tr\theta\cdot\dot{\beta}(t)\cos(\pi\epsilon-\beta(t))=-\tfrac{\det\theta}{2}\sin(2\beta(t))+\tr\theta\dot{\beta}(t)\cos\beta(t),$$
showing that $\beta$ is a solution of (\ref{HamiltonianEquivalent}). Since $\alpha$ is also a solution of (\ref{HamiltonianEquivalent}), we get by uniqueness that $\alpha=\beta$, as stated.
\end{proof}

Next we show that the period of a solution of (\ref{HamiltonianEquivalent}) can be recovered by looking at the times of intersection of the solutions with specific lines of the cylinder.



\begin{proposition}
\label{periodic}
      Let $\varphi\in \mathbb{S}^1$ be a solution of (\ref{HamiltonianEquivalent}) determined by $\lambda\in C$ and define the set 
      $$A(\lambda)=\{T\geq 0; 2\varphi(T)=\pi\epsilon, \;\epsilon=-1, 1\}.$$
      Then, $\varphi$ has finite period if and only if  $\lambda$ is not an equilibrium of $(\ref{HamiltonianEquivalent})$ and the cardinality of $A(\lambda)$ is at least two. Moreover in this case, 
      $$\tau(\lambda)=2\inf\{T_2-T_1; 0\leq T_1<T_2 \mbox{ with }T_1, T_2\in A(\lambda)\},$$
      and there exists $T_1\geq 0$ such that 
      $$2T_1<\tau(\lambda)\hspace{.5cm}\mbox{ and }\hspace{.5cm} A(\lambda)=\left\{T_1+ k\frac{\tau(\lambda)}{2}, k\geq 0\right\}.$$
\end{proposition}

\begin{proof}
   If the solution $\varphi$ has finite period, then the cardinality of $A(\lambda)$ is certainly greater or equal two. Reciprocally, if such cardinality is greater or equal two, let $T_1, T_2\in A(\lambda)$ be arbitrary times satisfying $0\leq T_1<T_2$. Using then relation (\ref{property}) for these times, allows us to get
    $$\forall t\in\R,\hspace{.5cm} \varphi(t+2(T_2-T_1))=\varphi(T_2+(t+T_2-2T_1))=\epsilon
    _2\pi-\varphi(-t+2T_1)$$
    $$=\epsilon
    _2\pi-\varphi(T_1+(-t+T_1))=\epsilon
    _2\pi-\left[\epsilon
    _1\pi-\varphi(t)\right]=\varphi(t),$$
    since $\varepsilon_i\in\{-1, 1\}$ for $i=1, 2$ and the sum is modulo $2\pi$. Therefore $\tau(\lambda)\leq 2(T_2-T_1)$, showing that $\varphi$ has finite period. On the other hand, since any $T\in A(\lambda)$ satisfies property (\ref{property}), it holds that
$$\pi\epsilon-\varphi\left(T\pm\frac{\tau(\lambda)}{2}\right)=\varphi\left(T\mp\frac{\tau(\lambda)}{2}\right)=\varphi\left(T\mp\frac{\tau(\lambda)}{2}\pm \tau(\lambda)\right)=\varphi\left(T\pm\frac{\tau(\lambda)}{2}\right)\hspace{.5cm}\implies \hspace{.5cm}2\varphi\left(T\pm\frac{\tau(\lambda)}{2}\right)=\pi\epsilon.$$
Consequently, if $T_1=\min A(\lambda)$, we obtain
$$T_1<\frac{\tau(\lambda)}{2}\hspace{.5cm}\mbox{ and }\hspace{.5cm} A(\lambda)=\left\{T_1+ n\frac{\tau(\lambda)}{2}, n\geq 0\right\},$$
and hence
$$\tau(\lambda)=2\inf\{T_2-T_1; 0\leq T_1<T_2 \mbox{ with }T_1, T_2\in A(\lambda)\},$$
concluding the proof.
\end{proof}

\subsection{Symmetries of the vertical part of the Hamiltonian system}

Our goal in this section is to show that, even in the presence of the perturbation term, the solutions of (\ref{horizontaldouble}) admit nontrivial groups of discrete symmetries. These symmetries induce corresponding symmetries on the group $G(\theta)$, which commute with the exponential map. 

\subsubsection{Symmetries of the phase space of the perturbed pendulum}

In the regular case, the phase portrait of the perturbed pendulum admits the following symmetries:
\[
\varepsilon_1:(\varphi, r)\mapsto (-\varphi, -r), \qquad 
\varepsilon_2:(\varphi, r)\mapsto (\varphi+\pi, -r), \qquad
\varepsilon_3:(\varphi, r)\mapsto (-\varphi+\pi, r).
\]

These transformations generate the discrete Klein four-group 
\[
\mathcal{K}=\{1, \varepsilon_1,\varepsilon_2,\varepsilon_3\}.
\] 
Among them, the map $\varepsilon_1$ preserves the orientation of time, while $\varepsilon_2$ and $\varepsilon_3$ reverse the direction of time in the solutions of the perturbed pendulum.

\subsubsection{Symmetries of the solutions of the pendulum}

We now show that the symmetries of the phase portrait of the perturbed pendulum induce corresponding symmetries of the solutions of (\ref{verticaldouble}).

\begin{proposition}
\label{verticalpart}
    The maps 
    \[
    \varepsilon_i: \{(\varphi(t), r(t)), \, t\in[0,T]\}\mapsto \{(\varphi_i(t), r_i(t)), \, t\in[0,T]\}, \qquad i=1,2,3,
    \]
    defined by
    \[
    (\varphi_1(t), r_1(t))=(-\varphi(t), -r(t)), \qquad
    (\varphi_2(t), r_2(t))=(\varphi(T-t)+\pi, -r(T-t)),
    \]
    \[
    (\varphi_3(t), r_3(t))=(-\varphi(T-t)+\pi, r(T-t)),
    \]
    transform trajectories of the perturbed pendulum (\ref{HamiltonianEquivalent}) into solutions of the same system.
\end{proposition}

\begin{proof}
Since the three cases are analogous, we show the claim only for $\varepsilon_3$. For this case, 
\[
\frac{d}{dt}\varphi_3(t)=\frac{d}{dt}\bigl(-\varphi(T-t)+\pi\bigr)=\dot{\varphi}(T-t)=r(T-t)=r_3(t),
\]
and
\[
\frac{d}{dt}r_3(t)=\frac{d}{dt}r(T-t)=-\dot{r}(T-t)=-\Bigl(-\tfrac{\det\theta}{2}\sin 2\varphi(T-t)+\tr\theta\, r(T-t)\cos\varphi(T-t)\Bigr).
\]
Using the identities $\sin(2(\pi-\alpha))=\sin(2\alpha)$ and $\cos(\pi-\alpha)=-\cos(\alpha)$, this becomes
$$
-\tfrac{\det\theta}{2}\sin\bigl(2(\pi-\varphi(T-t))\bigr)+\tr\theta\, r(T-t)\cos\bigl(\pi-\varphi(T-t)\bigr)
=-\tfrac{\det\theta}{2}\sin 2\varphi_3(t)+\tr\theta\, r_3(t)\cos\varphi_3(t),
$$
showing that $(\varphi_3(t), r_3(t))$ is indeed a solution of (\ref{HamiltonianEquivalent}), as claimed.
\end{proof}

\subsubsection{Symmetries of extremal trajectories}

The action of $\mathcal{K}$ on the solutions of the vertical part of the Hamiltonian can be extended to the solutions of the maximized Hamiltonian as follows:
$$\varepsilon_i:\{\lambda(t),\, t\in [0, T]\}\mapsto\{\lambda_i(t),\, t\in [0, T]\}.$$
The next result describes the relation between the normal extremal trajectory $\lambda$ and its image under the action of $\mathcal{K}$.

\begin{proposition}
\label{horizontalpart}
    Let $g(t)=(z(t), w(t))$, $t\in[0, T]$, be a normal extremal trajectory, and let $g_i(t)=(z_i(t), w_i(t))$, $t\in [0, T]$, be the trajectory obtained by the action of $\varepsilon_i$ defined above. Then:
    $$z_1(t)=z(t)\hspace{.5cm}\mbox{ and }\hspace{.5cm} w_1(t)=-w(t),$$
    $$z_2(t)=z(T-t)-z(T)\hspace{.5cm}\mbox{ and }\hspace{.5cm} w_2(t)=\rho_{-z(T)}\bigl(w(T-t)-w(T)\bigr),$$
    $$z_3(t)=z(T-t)-z(T)\hspace{.5cm}\mbox{ and }\hspace{.5cm} w_3(t)=\rho_{-z(T)}\bigl(w(T)-w(T-t)\bigr).$$
\end{proposition}

\begin{proof}
    As in the previous proposition, we consider only the case $\varepsilon_3$. We have
    $$z_3(t)=\int_0^t\cos\varphi_3(s)\,ds=\int_0^t\cos(-\varphi(T-s)+\pi)\,ds=-\int_0^t\cos\varphi(T-s)\,ds,$$
    which can be rewritten as
    $$-\int_{T-t}^T\cos\varphi(\tau)\,d\tau=-\int_0^T\cos\varphi(\tau)\,d\tau+\int_0^{T-t}\cos\varphi(\tau)\,d\tau=z(T-t)-z(T).$$
    Similarly,
    $$w_3(t)=\int_0^t\sin\varphi_3(s)\,\rho_{z_3(s)}\eta\,ds=\int_0^t\sin(-\varphi(T-s)+\pi)\,\rho_{z(T-s)-z(T)}\eta\,ds,$$
    $$=\rho_{-z(T)}\int_0^t\sin\varphi(T-s)\,\rho_{z(T-s)}\eta\,ds=\rho_{-z(T)}\int_{T-t}^T\sin\varphi(\tau)\,\rho_{z(\tau)}\eta\,d\tau.$$
    Hence,
    $$w_3(t)=\rho_{-z(T)}\left(\int_0^T\sin\varphi(\tau)\,\rho_{z(\tau)}\eta\,d\tau-\int_0^{T-t}\sin\varphi(\tau)\,\rho_{z(\tau)}\eta\,d\tau\right)=\rho_{-z(T)}\bigl(w(T)-w(T-t)\bigr),$$
    which proves the claim.
\end{proof}

\subsubsection{Symmetries of the endpoints of extremal trajectories}

Following the idea in \cite{MoSach, Sach1, Sach2}, we define the end-point action of $\varepsilon_i$ on $G(\theta)$ as 
$$\varepsilon_i:G(\theta)\rightarrow G(\theta), \hspace{.5cm}g(T)\mapsto g_i(T).$$
By Proposition \ref{horizontalpart}, the point $g_i(T)$ depends only on $g(T)$ and not on the whole trajectory $\{g(t), t\in [0, T]\}$. Since $G(\theta)$ is complete, this shows that $\varepsilon_i$ is a well-defined action on the group $G(\theta)$. The next result gives an explicit expression for these actions.  

\begin{proposition}
\label{endpointsreflection}
    Let $g=(z, w)\in G(\theta)$ and let $g_i=\varepsilon_i(g)$. Then:
    $$g_1=(z, -w), \hspace{.5cm} g_2=(-z, -\rho_{-z}w),\hspace{.5cm}\mbox{ and }\hspace{.5cm}
    g_3=(-z, \rho_{-z}w).$$
\end{proposition}

\begin{proof}
    The result follows directly by evaluating at $t=T$ in the formulas obtained in Proposition \ref{horizontalpart}.
\end{proof}

\subsubsection{Symmetries of the Exponential Map}

Finally, we define the action of $\mathcal{K}$ on the preimage of the exponential map as
$$\varepsilon_i:C\times\R_+\rightarrow C\times\R_+, \hspace{.5cm}(\lambda, T)\mapsto (\lambda_i, T),$$
where $\lambda$ and $\lambda_i$ are, respectively, the initial points of the trajectories $(\varphi(t), r(t))$ and $(\varphi_i(t), r_i(t))$ of the perturbed pendulum (\ref{verticaldouble}). Their explicit expressions are given in the following result.

\begin{proposition}
\label{igualpendulo}
Let $(\lambda, T)\in C\times\R_+$ and $(\lambda_i, T)\in C\times\R_+$. If $\lambda=(\varphi, r)\in \mathbb{S}^1\times\R$, then
$$\lambda_1=-\lambda, \hspace{.5cm}\lambda_2=(\varphi(T)+\pi, r(T)),\hspace{.5cm}\text{and}\hspace{.5cm}\lambda_3=(-\varphi(T)+\pi, r(T)).$$
\end{proposition}

\begin{proof}
Apply the formulas in Proposition \ref{verticalpart} at $t=0$.
\end{proof}

\bigskip

The results of the previous sections imply that the actions on the preimage of the exponential map commute with the endpoint action, i.e.,
$$\mathrm{Exp}\circ\varepsilon_i=\varepsilon_i\circ\mathrm{Exp}.$$
Indeed, for any $\lambda\in C$ and $T\in\R_+$, we have
$$\mathrm{Exp}\bigl(\varepsilon_i(\lambda, T)\bigr)=\mathrm{Exp}(\lambda_i, T)=\pi(\lambda_i(T))=g_i(T)=\varepsilon_i(g(T))=\varepsilon_i(\pi(\lambda(T)))=\varepsilon_i\bigl(\mathrm{Exp}(\lambda, T)\bigr).$$

\section{Maxwell points and optimality of normal extremal trajectories}

In this section, we introduce the notions of Maxwell points and Maxwell times, which play a fundamental role in the study of extremal trajectories. In the analytic setting, a normal trajectory ceases to be optimal after reaching a Maxwell point. Therefore, the corresponding Maxwell time provides an upper bound for the cut time, marking the loss of optimality and offering key insight into the trajectory's behavior.

\begin{definition}[Maxwell Point] 
A point $g(T)$, $T>0$, on a sub-Riemannian geodesic is called a \emph{Maxwell point} if there exists another extremal trajectory $\widetilde{g}$ with the same initial condition, such that $g(t)\neq \widetilde{g}(t)$ for all $t\in (0, T)$ and $g(T)=\widetilde{g}(T)$.
\end{definition}

For the symmetries in the group $\mathcal{K}$ associated with the solutions of (\ref{HamiltonianEquivalent}), we define the {\bf Maxwell stratum} $\mathrm{MAX}_i$ corresponding to the symmetry $\varepsilon_i$, $i=1,2,3$, as
$$
\mathrm{MAX}_i := \{ (\lambda, T) \in C \times \R_+ \;|\; \lambda \neq \lambda_i \text{ and } \mathrm{Exp}(\lambda, T) = \mathrm{Exp}(\lambda_i, T) \},
$$
and the {\bf Maxwell set} of $\mathcal{K}$ as $\mathrm{MAX} = \mathrm{MAX}_1 \cup \mathrm{MAX}_2 \cup \mathrm{MAX}_3$. The corresponding Maxwell strata and set at the group level are then
$$
\mathrm{Max}_i := \pi\bigl(\mathrm{MAX}_i\bigr), \quad i=1,2,3, \qquad \mathrm{Max} = \bigcup_{i=1}^3 \mathrm{Max}_i,
$$
respectively. If $(\lambda, T) \in \mathrm{MAX}_i$, then $g(T) = \mathrm{Exp}(\lambda, T) \in \mathrm{Max}_i$ is a Maxwell point on the geodesic $g(t) = \mathrm{Exp}(\lambda, t)$, since 
$$
\lambda \neq \lambda_i \quad \implies \quad \mathrm{Exp}(\lambda, t) \not\equiv \mathrm{Exp}(\lambda_i, t), \quad i=1,2,3.
$$

Since the elements of the $\mathcal{K}$-action on $G(\theta)$ commute with the exponential map, Proposition \ref{endpointsreflection} implies
\begin{equation}
\label{MaxwellinG}
g = \varepsilon_i(g) \iff
\begin{cases}
g \in \R \times \{0\}, & i=1,\\
g = (0,0), & i=2,\\
g \in \{0\} \times \R^2, & i=3.
\end{cases}
\end{equation}
Hence, the Maxwell points satisfy
$$
\mathrm{Max} \subset (\R \times \{0\}) \cup (\{0\} \times \R^2).
$$

\subsection{The zeros of $z(t)$ and $w(t)$}

By relation (\ref{MaxwellinG}), a necessary condition for a point $g=(z,w)$ to belong to a Maxwell stratum is that either $z=0$ or $w=0$. In this section, we analyze when a normal extremal trajectory starting at the origin satisfies this condition.

\subsubsection{The zeros of $w(t)$}

Since $T^*G(\theta)$ is trivializable, we can assume that any element in $T^*G(\theta)$ is written as 
$$\lambda=((\mu, p), (z, w)), \hspace{.5cm}\mbox{ where }\mu\in\R, p\in(\R^2)^*, \mbox{ and }g=(z, w)\in G(\theta). $$

In particular, we get that
$$h_1(\lambda)=\mu,\hspace{.5cm}h_2(\lambda)=\langle p, \rho_z\eta\rangle, \hspace{.5cm}\mbox{ and }\hspace{.5cm} H(\lambda)=\frac{1}{2}(\mu^2+\langle p, \rho_z\eta\rangle^2).$$
In these coordinates, the Hamiltonian equations are written as
$$\displaystyle\left\{\begin{array}{l}
\dot{z}=\frac{\partial H}{\partial \mu}=\mu\\
   \\
\dot{w}=\frac{\partial H}{\partial p}=\langle p, \rho_z\eta\rangle \rho_z\eta
\end{array}\right.\hspace{.5cm}\mbox{ and }\hspace{.5cm}\left\{\begin{array}{l}
\dot{\mu}=-\frac{\partial H}{\partial z}=-\langle p, \rho_z\eta\rangle \langle p, \rho_z\theta\eta\rangle 
\\

\\
\dot{p}=-\frac{\partial H}{\partial w}=0
\end{array}\right..$$

The next result shows that the planar component of normal extremal trajetories is constant equals zero or have no zeros for positive times.

\begin{proposition}
\label{zerosw}
    Let $\lambda\in C$ and consider $g(t)=(z(t), w(t))$ the associated normal extremal trajectory. Then,  $w(t)\neq 0$ for all $t>0$ or $w\equiv 0$.
\end{proposition}

\begin{proof}
By the previous, if $\lambda=(\mu, p)\in C$, then the second component of the vertical part of the Hamiltonian equation implies that $p(t)=p$ is constant. Hence, 
     $$\sin\varphi(t)=\langle\lambda(t), (0, \rho_z\eta)\rangle=\langle p, \rho_{z(t)}\eta\rangle\hspace{.5cm}\implies \hspace{.5cm}\langle p, \dot{w}\rangle=\langle p, \rho_{z(t)}\eta\rangle^2=\sin^2\varphi(t)\geq 0,$$
     showing that the function $t\mapsto \langle p_0, w(t)\rangle$ is nondecreasing. Since, $w(0)=0$, the existence of $T>0$ with $w(T)=0$ imply that 
     $$0=\langle w(0), p\rangle \leq \langle w(t), p\rangle\leq \langle  w(T), p\rangle=0, \hspace{.5cm}\forall
    t\in [0,T],$$ 
    proving $\varphi(t)\equiv k\pi$ for all $t\in [0, T]$. Therefore, one must $\lambda$ is a singularity and $w\equiv 0$, as stated.
\end{proof}

\subsubsection{The zeros of $z(t)$} 

The next result relates property (\ref{property}) of the solutions of the vertical part with the zeros of the $z$-component of a normal extremal trajectory.

\begin{lemma}
\label{firstzero}
Let $\varphi\in \mathbb{S}^1$ be a solution of (\ref{HamiltonianEquivalent}) determined by $\lambda\in C$. If $T_0\in A(\lambda)$, then $z(2T_0)=0$. 
\end{lemma}

\begin{proof}
Since $T_0\in A(\lambda)$ it satisfies 
$$\forall t\in\R, \hspace{.5cm}\varphi(T_0-t)+\varphi(T_0+t)=\pi\epsilon.$$
Then, 
    $$\forall t\in\R, \hspace{.5cm}\cos\varphi(t)=\cos\varphi(T_0+(t-T_0))=\cos(\pi\epsilon-\varphi(T_0-(t-T_0)))=-\cos\varphi(-t+2T_0).$$
    Hence, using the previous relation and a change of variables, we get 
    $$\int_{T_0}^{2T_0}\cos\varphi(t)dt=-\int_{T_0}^{2T_0}\cos\varphi(-t+2T_0)dt=\int_{T_0}^0\cos\varphi(s)ds=-\int_0^{T_0}\cos\varphi(s)ds,$$
    implying that 
$$z(2T_0)=\int_{0}^{2T_0}\cos\varphi(t)dt=\int_{0}^{T_0}\cos\varphi(t)dt+\int_{T_0}^{2T_0}\cos\varphi(t)dt=0,$$
concluding the proof.
\end{proof}

We can now prove our main result of this section related to the zeros of the $z$-component of normal extremal trajectories. This result will be central in the proof for the upper bound of the cut time in the next section.

\begin{theorem}
\label{mainsection3}
Let $\varphi\in \mathbb{S}^1$ be a solution of (\ref{HamiltonianEquivalent}) determined by $\lambda\in C$. If $\lambda$ is not an equilibrium of (\ref{HamiltonianEquivalent}) it holds:
\begin{enumerate}
    \item The cardinality of $A(\lambda)$ is one and $z$ has at most one zero in $(0, +\infty)$;
    \item The cardinality of $A(\lambda)$ is greater or equal two and the zeros of $z$ are given by 
    $$2A(\lambda)\cup\N\tau(\lambda);$$
    \item If $\varphi$ is a periodic function with $2\varphi(0)\neq\pm\pi$, then 
    $$\tau(\lambda)=\min\Bigl\{T>0;\hspace{.5cm} z(T)=0\hspace{.3cm}\mbox{ and }\hspace{.5cm}\varphi(0)+\varphi(T)\neq\pm\pi\Bigr\}.$$
\end{enumerate}
\end{theorem}

\begin{proof} 1. Let us assume that $0<T_1<T_2$ be such that $z(T_1)=Z(T_2)=0$. Then, by Rolle's Theorem, 
$$z(0)=z(T_1)=z(T_2)=0\hspace{.5cm}\implies\hspace{.5cm} \exists \tau_1\in(0, T_1),\hspace{.5cm} \tau_2\in (T_1, T_2);  \hspace{.5cm}\dot{z}(\tau_1)=\dot{z}(\tau_2)=0.$$
Since $\dot{z}=\cos\varphi$, we get
$$i=1, 2, \hspace{.5cm}0=\dot{z}(\tau_i)=\cos\varphi(\tau_i)\hspace{.5cm}\implies\hspace{.5cm}2\varphi(\tau_i)=0\hspace{.5cm}\implies\hspace{.5cm} 2\varphi(\tau_i)=\pi\epsilon_i, \epsilon_i\in\{-1, 1\}.$$
Hence, $0<\tau_1<\tau_2$ and $\tau_1, \tau_2\in A(\lambda)$, showing that $A(\lambda)$ has cardinality at least two, which implies the result.

2. Since $A(\lambda)$ has cardinality at least two, Proposition \ref{periodic} implies the existence of $0\leq 2T_1<\tau(\lambda)$ such that 
$$A(\lambda)=\left\{T_1+k\frac{\tau(\lambda)}{2}, k\geq 0\right\}.$$
In particular, by Lemma \ref{firstzero}, we have that
$$\int_0^{2T_1}\cos\varphi(t)dt=\int_0^{2T_1+k\tau(\lambda)}\cos\varphi(t)dt=0,$$
hence, 
$$0=\int_0^{2T_1+k\tau(\lambda)}\cos\varphi(t)dt=\int_0^{2T_1}\cos\varphi(t)dt+\int_{2T_1}^{2T_1+k\tau(\lambda)}\cos\varphi(t)dt=\int_{2T_1}^{2T_1+k\tau(\lambda)}\cos\varphi(t)dt$$
$$=-\int_{2T_1}^{2T_1+k\tau(\lambda)}\cos\varphi(-t+2T_1+k\tau(\lambda))dt=\int_{0}^{k\tau(\lambda)}\cos\varphi(s)ds=z(k\tau(\lambda)),$$
where for the first equality in the second line we used property (\ref{property}) and the $\tau(\lambda)$-periodicity of $\varphi$.

Therefore, by Lemma 4.3 and the $\tau(\lambda)$-periodicity of the solution, we get that $2A(\lambda)\cup \N\tau(\lambda)$ are zeros of the $z$-component. Now, if $T>0$ is such that $z(T)=0$, there exists $n\in\N_0$ such that $T-n\tau(\lambda)\in [0, \tau(\lambda))$ and hence, it is enough to show that 
$$z(S)=0, \hspace{.5cm}S\in [0, \tau(\lambda))\hspace{.5cm}\implies \hspace{.5cm}S=2T_1,$$

Arguing by contradiction, if $S\in [0, \tau(\lambda))$ is such that $S\neq 2T_1$, we can as in item 1. apply Rolle's Theorem to obtain points $T_2, T_3$ in the intervals determined by $\{S, 2T_1\}$, and $\{2T_1, \tau(\lambda)\}$, respectively, such that  $T_2, T_3\in A(\lambda)$. Hence, $T_1, T_2, T_3$ are distinct points in $A(\lambda)\cap [0, \tau(\lambda))$ which, Proposition \ref{periodic}, is a contradiction as
$$A(\lambda)=\left\{T_1+k\frac{\tau(\lambda)}{2}, k\geq 0\right\}\hspace{.5cm}\implies\hspace{.5cm}A(\lambda)\cap[0, \tau(\lambda))=\left\{T_1, T_1+\frac{\tau(\lambda)}{2}\right\}.$$
Therefore, $2T_1$ is the only zero of $z$ in the interval $[0, \tau(\lambda))$, implying the $2A(\lambda)\cup\N\tau(\lambda)$ is the set of zeros of the $z$-component.

3. By the previous item, the first two positive zeros of $z$ are $2T_1$ and $\tau(\lambda)$ where $T_1=\min A(\lambda)$. Now, the fact that $T_1\in A(\lambda)$ implies by property (\ref{property}) that $\varphi(0)+\varphi(2T_1)=\epsilon\pi$ for $\epsilon\in\{-1, 1\}$. Hence, 
$$\varphi(0)+\varphi(T)=\epsilon\pi=\epsilon\pi+2(1-\epsilon)\pi=\pi,$$
implying that 
$$2T_1\notin \Bigl\{T>0;\hspace{.5cm} z(T)=0\hspace{.3cm}\mbox{ and }\hspace{.5cm}\varphi(0)+\varphi(T)\neq\pi\Bigr\}.$$
On the other hand, if $2\varphi(0)\neq\pm\pi$ then $\varphi(0)+\varphi(\tau(\lambda))=2\varphi(0)\neq\pm\pi$ showing that 
$$\tau(\lambda)\in\Bigl\{T>0; \hspace{.5cm}z(T)=0\hspace{.3cm}\mbox{ and }\hspace{.5cm}\varphi(0)+\varphi(T)\neq\pm\pi\Bigr\},$$
and concluding the result.
\end{proof}

\subsection{The Maxwell time and an upper bound for the cut time}

In this section we analyze the times associated with the Maxwell points. By using the results in the previous sections we are able to show that the first Maxwell time of the Klein four-group $\mathcal{K}$ given by the symmetries of (\ref{verticaldouble}) are constant of the solutions of the vertical part. Moreover, it coincides with the period of the pendulum for most of their solutions. Using that, we are able to conclude obtain that the period of the pendulum is an upper bound for the cut time of the normal extremal trajectories.

\begin{definition}
    For a normal extremal trajectory $g(t)$ with initial covector $\lambda$, we define:
    \begin{itemize}
\item[(i)] The cut time of $g(t)$ as
    $$t_{\mathrm{cut}}(\lambda):=\sup\{T>0; g(t)\mbox{ is optimal for }t\in [0, T] \}.$$

    \item[(ii)] The first Maxwell time as 
    $$t^{\mathrm{MAX}}_1(\lambda):=\inf\{T>0; (\lambda, T)\in\mathrm{MAX}\}.$$
    \end{itemize}
    
\end{definition}

Since we are in the analytic case, a normal extremal trajectory cannot be optimal after a Maxwell point, and hence
\begin{equation}
    \label{inequality}
    t_{\mathrm{cut}}(\lambda)\leq t^{\mathrm{MAX}}_1(\lambda), \hspace{.5cm}\forall \lambda\in C.
\end{equation}

We can now prove the main results of this paper.

\begin{theorem}
\label{mainMaxwell}
For the group of symmetries $\mathcal{K}$ of the pendulum (\ref{HamiltonianEquivalent}), it holds that
$$\mathrm{MAX}=\left\{(\lambda, n\tau(\lambda)); \;n\in\N, \hspace{.5cm} \tau(\lambda)<+\infty\hspace{.5cm}\mbox{ and }\hspace{.5cm} \lambda\notin\left\{\pm\frac{\pi}{2}\right\}\times\R\right\}$$
 $$\mbox{ and }\hspace{.5cm} t_1^{\mathrm{MAX}}(\lambda)=\left\{\begin{array}{cc}
        \tau(\lambda) & \lambda\notin\left\{\pm\frac{\pi}{2}\right\}\times\R, \\
        +\infty & \mbox{ otherwise.} 
    \end{array}\right.$$

\end{theorem}

\begin{proof} By Proposition \ref{zerosw} the planar component of a normal extremal trajectory is never zero, or it is trivial. Hence, $\mathrm{MAX}_1=\emptyset$ implying that $\mathrm{MAX}=\mathrm{MAX}_3$, and we only have to analyze characterize the set $\mathrm{MAX}_3$.

If $\lambda\in C$ is an equilibrium point of (\ref{HamiltonianEquivalent}), the normal extremal trajectory can be explicitly calculated from the equations (\ref{horizontaldouble}) and takes the form 
$$g(t)=(0, \pm t\eta), \hspace{.5cm}\lambda\in \{p_2, p_4\}\hspace{.5cm}\mbox{ or }\hspace{.5cm}g(t)=(\pm t, 0) \hspace{.5cm}\lambda\in \{p_1, p_3\}.$$
Therefore, if $\lambda\in \{p_1, p_3\}$ then $(\lambda, T)\notin\mathrm{MAX}$ for all $T>0$. On the other hand, if $\lambda\in \{p_2, p_4\}$, then for all $T>0$ and some $\epsilon\in\{-1, 1\}$,
$$\varphi(0)+\varphi(T)=\epsilon\pi=\epsilon\pi+2(1-\epsilon)\pi=\pi\hspace{.5cm}\implies\hspace{.5cm}\lambda=(\varphi(0), 0)=(-\varphi(T)+\pi, 0)=\lambda_3,$$
showing that $(\lambda, T)\notin\mathrm{MAX}$ for all $T>0$. Therefore, if $(\lambda, T)\in \mathrm{MAX}$ then $\lambda$ is not an equilibrium of (\ref{HamiltonianEquivalent}).

Let then $(\lambda, T)\in\mathrm{MAX}_3$. By the equations (\ref{MaxwellinG}), it holds that 
$$\mathrm{MAX}_3=\{(\lambda, T)\in C\times\R_+; \hspace{.5cm}\lambda\neq\lambda_3\hspace{.5cm}\mbox{ and } \hspace{.5cm} z(T)=0\}.$$
Since $z(T)=0$, Rolle's Theorem implies that $A(\lambda)\neq\emptyset$. Moreover, since $\lambda$ is not an equilibrium, Theorem \ref{mainsection3} implies that $T\in 2A(\lambda)$ if the cardinality of $A(\lambda)$ is one and $T\in 2A(\lambda)\cup \N\tau(\lambda)$ if the cardinality is greater or equal two. As in the proof of item 3. of Theorem \ref{mainsection3}, if $T\in 2A(\lambda)$, then $\varphi(0)=-\varphi(T)+\pi$. Moreover, since $T/2$ satisfies property (\ref{property}), we get by derivation that 
$$ \forall t\in \R, \hspace{.5cm}r\left(\frac{T}{2}-t\right)=r\left(\frac{T}{2}+t\right)\hspace{.5cm}\implies\hspace{.5cm} r(T)=r(0),$$
and hence  
$$\lambda=(\varphi(0), r(0))=(-\varphi(T)+\pi, r(T))=\lambda_3,$$
implying that $T\in\N\tau(\lambda)$. Furthermore, the fact that $\varphi$ is a $\tau(\lambda)$-periodic function and $T=n\tau(\lambda)$ implies in particular that 
$$\varphi(0)=\varphi(T)\hspace{.5cm}\mbox{ and }\hspace{.5cm} r(0)=r(T),$$
and hence 
$$\lambda\neq\lambda_3\hspace{.5cm}\implies \hspace{.5cm}\lambda\notin\left\{\pm\frac{\pi}{2}\right\}\times\R,$$
allowing us to conclude that
$$\mathrm{MAX}=\mathrm{MAX}_3\subset \left\{(\lambda, n\tau(\lambda)); \;n\in\N, \hspace{.5cm} \tau(\lambda)<+\infty\hspace{.5cm}\mbox{ and }\hspace{.5cm} \lambda\notin\left\{\pm\frac{\pi}{2}\right\}\times\R\right\}.$$

Reciprocally, let $\lambda\in C$ be such that $\tau(\lambda)<+\infty$ and $\lambda\notin \left\{\pm\frac{\pi}{2}\right\}\times\R$. If $\varphi$ is the solution of (\ref{HamiltonianEquivalent}) determined by $\lambda$, then $\varphi$ is a $\tau(\lambda)$-periodic curve and  
$$2\varphi(0)\neq\pm\pi\hspace{.5cm}\implies\hspace{.5cm}\varphi(0)+\varphi(n\tau(\lambda))\neq\pi\hspace{.5cm}\implies\hspace{.5cm}\lambda\neq\lambda_3.$$
Since by Theorem \ref{mainsection3} $z(n\tau(\lambda))=0$, we get the equality 
$$\mathrm{MAX}=\left\{(\lambda, n\tau(\lambda)); \;n\in\N, \hspace{.5cm} \tau(\lambda)<+\infty\hspace{.5cm}\mbox{ and }\hspace{.5cm} \lambda\notin\left\{\pm\frac{\pi}{2}\right\}\times\R\right\},$$
as desired. Furthermore,  
$$t_1^{\mathrm{MAX}}(\lambda)=\left\{\begin{array}{cl}
    \tau(\lambda) & \mbox{ if }\tau(\lambda)<+\infty\mbox{ and }\lambda\notin\left\{\pm\frac{\pi}{2}\right\}\times\R \\
    +\infty & \mbox{ if }\tau(\lambda)=+\infty\mbox{ or }\lambda\in\left\{\pm\frac{\pi}{2}\right\}\times\R
\end{array}\right.\hspace{.5cm}\iff\hspace{.5cm} t_1^{\mathrm{MAX}}(\lambda)=\left\{\begin{array}{cl}
    \tau(\lambda) & \mbox{ if }\lambda\notin\left\{\pm\frac{\pi}{2}\right\}\times\R \\
    +\infty & \mbox{ otherwise }
\end{array}\right.,$$
concluding the proof.
\end{proof}

\begin{corollary}
    The Maxwell time $t_1^{\mathrm{MAX}}$ is invariant by the solutions of (\ref{HamiltonianEquivalent}) and by the action of $\mathcal{K}$.
\end{corollary}

Inequality (\ref{inequality}) together with the previous theorem implies that the period of the pendulum is a natural upper bound for the cut time for all $\lambda\in C$ whose first component is not a multiple of $\pi/2$. In what follows we show that the previous holds for any point of the cylinder. In order to do that, we introduce the concept of conjugated point.

\begin{definition}
    We say that a point $g(T)$, of a strictly normal geodesic, 
    $$g(t)=\mathrm{Exp}(\lambda, t), \hspace{.5cm} t\in [0, T],$$
    is a {\bf conjugate} to the point $g(0)$ along the geodesic $g(t)$ if $(\lambda, T)$ is a critical point of the exponential mapping. 
\end{definition}

By the general theory, a strictly normal geodesic cannot be optimal after a conjugate point (see for instance \cite[Chapter 8]{AADBUB}). In particular, the time associated to conjugate points are upper bounds for the cut time. The next result from asserts that the limit of Maxwell points are a conjugate point (see \cite[Proposition 5.1]{Sach4}).

\begin{proposition}
    Let $(\lambda_n, T_n),(\lambda'_n, T'_n)\in C\times\R_+$ and assume that $\mathrm{Exp}(\lambda_n, T_n)=\mathrm{Exp}(\lambda'_n, T'_n)$ for all , $n\in\N$. If $(\lambda_n, T_n)\neq(\lambda'_n, T'_n)$ and both sequences converge to the same point $(\lambda, T)$ and the geodesic $g(t)=\mathrm{Exp}(\lambda, t)$ is strictly normal, then its endpoint $g(T)=\mathrm{Exp}(\lambda, T)$ is a conjugate point.
\end{proposition}

We can now prove the following:

\begin{theorem}
    It holds
    $$t_{\mathrm{cut}}(\lambda)\leq \tau(\lambda), \hspace{.5cm}\forall \lambda\in C.$$
\end{theorem}

\begin{proof}
    From Theorem \ref{mainMaxwell} and inequality (\ref{inequality}) we only have to show the result for $\lambda\in\left\{\pm\frac{\pi}{2}\right\}\times\R$. Moreover, if $\tau(\lambda)=+\infty$ the result is certainly true, so we only have to consider 
    $$\lambda\in\left\{\pm\frac{\pi}{2}\right\}\times\R\hspace{.5cm}\mbox{ and }\hspace{.5cm}\tau(\lambda)<+\infty.$$
Under these hypothesis, the solution $\varphi$ of (\ref{HamiltonianEquivalent}) determined by $\lambda$ is a periodic solution. By continuity dependence of initial conditions, there exists periodic solutions $\varphi_n\in C$ of (\ref{HamiltonianEquivalent}) determined by $\lambda_n\in C$ such that 
$$\lambda_n\notin\left\{\pm\frac{\pi}{2}\right\}\times\R, \hspace{.5cm}\lambda_n\rightarrow\lambda, \hspace{.5cm}\mbox{ and }\hspace{.5cm}\tau(\lambda_n)\rightarrow\tau(\lambda).$$
Since $\tau(\lambda_n)<+\infty$ we conclude that $(\lambda_n, \tau(\lambda_n))\in\mathrm{MAX}$. In particular, $(\lambda_n, \tau(\lambda_n))\neq \varepsilon_3(\lambda_n, \tau(\lambda_n))$ and both sequences converge to $(\lambda, \tau(\lambda))$. Moreover, by the previous sections,  $\mathrm{Exp}(\lambda_n, \tau(\lambda_n))=\mathrm{Exp}(\varepsilon_3(\lambda_n, \tau(\lambda_n)))$, which by the previous proposition implies that $(\lambda, \tau(\lambda))$ is a conjugate point, and hence, 
$$t_{\mathrm{cut}}(\lambda)\leq \tau(\lambda), \hspace{.5cm}\forall \lambda\in C,$$
as stated.
\end{proof}

\section{Conclusion and future work}

The previous results show that, for almost all geodesics, the first Maxwell time is given by the period (in the extended sense) of the solutions of the vertical part. As a consequence, this period provides an upper bound for the cut time of any geodesic. A natural next question is therefore to determine when the cut time actually coincides with the period.

In our setting, the cut time is given by the minimum between the first Maxwell time and the first conjugate time. Thus, addressing this question reduces to the study of conjugate points for the sub-Riemannian problem on regular three-dimensional Lie groups. For the three-dimensional solvable Lie groups \(SE(2)\) and \(SH(2)\), it was shown (as a byproduct) that the cut time does indeed coincide with the period; both results can be found in \cite{Sach5}. We believe that the same conclusion holds for some of the regular cases considered here, depending on the eigenvalues of the matrix \(\theta\).

\end{document}